\documentclass[12pt]{article}
\RequirePackage{amsthm,amsmath,amssymb,amscd}
\usepackage{verbatim}
\usepackage{color}
\usepackage{float} 
\usepackage[active]{srcltx}
\usepackage[usenames,dvipsnames]{xcolor}

\allowdisplaybreaks[4]

\theoremstyle{plain}
\newtheorem{theorem}{Theorem}[section]

\newtheorem{lemma}[theorem]{Lemma}

\theoremstyle{definition}
\newtheorem{definition}[theorem]{Definition}

\let\Section=\section
\def\section{\setcounter{equation}{0}\Section}
\newcommand{\R}{\mathbb{R}}

\def\RR{\mathbb{R}}

\def\EE{\mathbb{E}}

\def \eref#1{\hbox{(\ref{#1})}}

\usepackage{booktabs}

\newcommand{\Ceil}[1]{\left\lceil #1 \right\rceil}

\usepackage{graphicx}
\usepackage{graphics}
\usepackage{epsfig}

\newcommand{\ud}{\ensuremath{\mathrm{d}}}

\RequirePackage[colorlinks,citecolor=blue,urlcolor=blue]{hyperref}
\usepackage{enumerate}
\newenvironment{keywords}{\small\begin{quote}{\textbf{Keywords:}}\,\,}{\end{quote}}

\begin{document}

\title{Some stochastic time-fractional diffusion equations with variable coefficients and time dependent noise} 

\author{   {\sc Guannan Hu}}
\date{}
\maketitle

\begin{abstract}
We prove the existence and uniqueness of mild solution for the stochastic partial differential equation  $$\left(\partial^\alpha   - \textit{B} \right) u(t,x)= u(t,x) \cdot \dot{W}(t,x),$$  where $$\alpha \in (1/2, 1)\cup(1, 2);$$ $\textit{B}$ is an uniform elliptic operator  with variable coefficients and  $\dot W$ is a Gaussian noise general in time with  space covariance given by fractional, Riesz and Bessel kernel.
\end{abstract}

\begin{keywords}
Gaussian noisy environment, time fractional order
spde, Fox H-functions, mild solutions, uniform elliptic operator, chaos expansion, Riesz kernel,  Bessel kernel.
\end{keywords}

\section{Introduction}
In this article we prove the existence and uniqueness of the mild solution of  the equation

\begin{align}  \label{E:SPDE}
\begin{cases}
\displaystyle \left(\partial^\alpha  - \textit{B}\right) u(t,x)=u(t,x) \dot{W}(t,x),&\qquad t\in(0, T],\: x\in\RR^d, \\[0.5em]
\displaystyle \left.\frac{\partial^k}{\partial t^k} u(t,x)\right|_{t=0}=u_k(x), &\qquad 0\le k\le \Ceil{\alpha}-1, \:\: x\in\RR^d,
\end{cases}
\end{align}
with any fixed $T\in \RR^+$, $\alpha \in (1/2, 1) \cup (1, 2)$, where $\Ceil{\alpha}$ is the smallest integer not less than $\alpha$.
Here we assume 
\begin{itemize}
\item $u_0(x)$ is bounded continuously differentiable. Its first order derivative bounded and H\"{o}lder continuous. The  H\"{o}lder exponent $\gamma >\frac{2-\alpha}{\alpha}$ 
\item  $u_1(x)$ is bounded continuous function(locally h\"{o}lder continuous if $d>1$)
\end{itemize}
In this equation, $\dot{W}$ is a  zero mean   Gaussian noise with the following  covariance structure
\[
\EE(\dot{W}(t,x)\dot{W}(s,y))=\lambda(t-s)\Lambda(x-y),
\]
where  $\lambda(\cdot)$ is nonnegative definite and locally intergrable and $\Lambda(\cdot)$ is one of the following situations:
\begin{enumerate}[(i)]

\item Fractional kernel. $\displaystyle \Lambda(x):=\prod^{d}_{i=1}2H_i(2H_i-1)|x_i|^{2H_i-1},$ $x\in\RR^d$ and $1/2<H_i<1$.
\item Reisz kernel. $\Lambda(x):=C_{\alpha,d}|x|^{-\kappa}$, $x\in\RR^d$ and $0<\kappa<d$ and $C_{\alpha,d}=\Gamma(\frac\kappa2)2^{-\alpha}\pi^{-d/2}/\Gamma(\frac\alpha2).$
\item Bessel kernel. $\Lambda(x):=C_\alpha\int_0^\infty \omega^{-\frac{\kappa}{2}-1}e^{-\omega} e^{\frac{-|x|^2}{4\omega}}d\omega$, $x\in\RR^d$,  $0<\kappa<d$, and $C_\alpha=(4\pi)^{\alpha/2}\Gamma(\alpha/2)$.
\end{enumerate}

$$\textit{B}:=\sum^d_{i, j=1}a_{i, j}(x)\frac{\partial^2}{\partial x_i\partial x_j}+\sum^d_{j=1}b_j(x)\frac{\partial}{\partial x_j}+c(x)$$ is uniformly elliptic. Namely it satisfies the following conditions:
\begin{enumerate}[(i)]
\item $a_{ij}(x),  b_j(x)$ and $c(x)$ are bounded H\"older continuous functions on $\RR^d$
\item $\exists a_0>0,$ such that $\forall x, \xi \in \RR^d,$ 
$$\sum^d_{i, j=1}a_{i, j}(x)\xi_i\xi_j \geq a_0 |\xi|^2.$$
\end{enumerate}
The fractional derivative in time $\partial^\alpha  $ is understood in {\it Caputo} sense: 
\[
\partial^\alpha  f(t) :=
\begin{cases}
\displaystyle
 \frac{1}{\Gamma(m-\alpha)} \int_0^t\ud
\tau\: \frac{f^{(m)}(\tau)}{(t-\tau)^{\alpha+1-m}}&
\text{if $m-1<\alpha<m$\;,}\\[1em]
\displaystyle
\frac{\ud^m}{\ud t^m}f(t)& \text{if $\alpha=m$}\;.
\end{cases}
\]
Throughout this chapter, the initial conditions $u_k(x)$ are bounded continuous(H\"older continuous, if $d>1$) functions.
The study of the mild solution relies on the asymptote property of the Green's function $Z, Y$ of the following deterministic equation.

\begin{align} \label{E:PDE}
\begin{cases}
\displaystyle \left(\partial^\alpha  - {\it B} \right) u(t,x)= f(t,x),&\qquad t>0,\: x\in\RR^d, \\
\displaystyle \left.\frac{\partial^k}{\partial t^k} u(t,x)\right|_{t=0}=u_k(x), &\qquad 0\le k\le \Ceil{\alpha}-1, \:\: x\in\RR^d,
\end{cases}
\end{align}
In  \cite{HuHu15} we cover the case $\alpha\in (1/2, 1)$.
When $\alpha\in (1, 2)$, \cite{Pskhu09} showed that  when {\it B} is $\Delta$, Green's function $Y$ of \eqref{E:PDE} the following:
$$Y(t, x)=C_dt^{\frac\alpha2(2-d)}f_{\frac\alpha2}(|x|t^{-\frac\alpha2}; d-1, \frac\alpha2(2-d)),$$
where  
$$f_{\frac\alpha2}(z; \mu, \delta)=
\begin{cases}
\displaystyle  \frac{2}{\Gamma(\frac{\mu}{2})}\int_1^{\infty}\phi(-\frac\alpha2, \delta; -zt)(t^2-1)^{\frac{\mu}{2}-1}\ud t, &\qquad \mu>0, \\[0.5em]
\displaystyle \phi(-\frac\alpha2, \delta; -z), &\qquad \mu=0;
\end{cases}
$$ $C_d=2^{-n}\pi^{\frac{1-d}{2}}$ and  the wright's function
$$\phi(-\frac\alpha2, \delta; -z):=\sum_{n=0}^{\infty}\frac{z^n}{n!\Gamma(\delta-\frac\alpha2 n)}$$

The solution of \eqref{E:PDE} has the following form:
\begin{align}\label{E:Duhamel}
 u(t,x) = J_0(t,x) +
\int_0^t \ud s \int_{\RR^d} \ud y\: f(s,y) Y(t-s,x-y),
\end{align}
where and throughout the chapter, we denote
\begin{align}\label{E:J0}
J_0(t,x):=
\sum_{k=0}^{\Ceil{\alpha}-1}\int_{\RR^d} u_{ k}(y) Z_{k+1}(t,x-y) \ud y\,.
\end{align}
For case of $\alpha\in(1/2, 1)$, we use $Z$ in place of $Z_1$.
We have the following facts about $Z_1(t, x),\,  Z_2(t, x)$ and $Y(t, x)$.
$$Z_1(t, x)=D^{\alpha-1}Y(t, x); \qquad  Z_1(t, x)=\frac{\partial }{\partial t}Z_2(t, x)$$

As in \cite{HuHu15},  We first get the estimation of $Y$, then use Wiener chaos expansion to obtain relation between the parameter $\alpha, d, H_i$ and $\kappa$ such that the mild solution exist.

The rest of the article is organized as follows. Section 2  gives more details about the solution of \eqref{E:SPDE}, estimation of $Y$ for $\alpha\in(1/2, 1)$ and some preliminaries about Wiener spaces.  Section 3 gives the estimation of $Y$ for $\alpha\in(1, 2)$ and further estimations before proving the existence of the mild solution.
\\
{\bf Notation:} Throughout this chapter we denote $$p(t,x):=\exp\left\{-\sigma\left|\frac{x}{t^{\frac{\alpha}{2}}}\right|^{\frac{2}{2-\alpha}}\right\},$$ where
$\sigma>0$ is a generic positive constant whose values may vary at different occurrence, so is C.

\bigskip

\section{Preliminary}\label{Sec:Pre}
We consider a Gaussian noise $W$ on a complete probability space
$(\Omega,\mathcal{F},P)$ encoded by a
centered Gaussian family $\{W(\varphi) ; \, \varphi\in
L^2(\R_+\times \R^{d})\}$, whose covariance structure $\lambda(s-t)$
is given by
\begin{equation}\label{cov1}
\EE \left (  W(\varphi) \, W(\psi) \right)
= \int_{\R_{+}^{2}\times\R^{2d}}
\varphi(s,x)\psi(t,y)\lambda(s-t)\Lambda(x-y)\ud x\ud y\ud s\ud t,
\end{equation}
where $\lambda: \R \rightarrow \R_+$  is nonnegative definite and locally intergrable.
Throughout the chapter, we denote
\begin{align}\label{E:Ct}
 C_t := 2 \int_0^t \lambda(s)\ud s, \quad t>0.
\end{align}
$\Lambda: \R^d \rightarrow
\R_+$ is a fractional, Reisz or Bessel kernel.

\begin{definition}\label{D:Sol} Let $Z$ and $Y$ be the fundamental solutions defined by
\eqref{E:PDE} and \eqref{E:Duhamel}.
An adapted random field $ \{ u={u(t,x): \:t\geq 0, x\in \mathbb{R}^d} \} $ such that
$\EE \left[u^2(t,x)\right]<+\infty$ for all $(t,x)$
is a {\it mild solution} to \eqref{E:SPDE}, if for all $(t,x)\in\R_+\times \mathbb{R}^d$,
the process
\[
\left\{Y(t-s, x-y)u({s,y})1_{[0,t]}(s): \: s\ge0,\: y\in \mathbb{R}^d \right\}
\]
 is Skorodhod integrable (see \eqref{E:dual}), and
 $u$ satisfies
 \begin{equation}\label{E:mild}
u(t,x)=J_0(t,x)+\int_0^t\int_{\RR^d}Y(t-s,x-y)u(s,y)  W(\ud s,\ud y)
\end{equation}
 almost surely for all $(t,x)\in\RR_+\times\RR^d$,  where $J_0(t,x)$ is defined by
 \eqref{E:J0}.
\end{definition}

We use a similar chaos expansion to the one used in chapter 3. To prove the existence and uniqueness of the solution we show that for all $(t,x)$,
\begin{equation}\label{eq: L2 chaos}
\sum_{n=0}^{\infty}n!\|f_n(\cdot,\cdot,t,x)\|^2_{\mathcal{H}^{\otimes n}}< \infty\,.
\end{equation}

\section{Estimations of the Green's functions}
The fundamental solution of \eqref{E:PDE} is constructed by Levi's parametrix method.  We refer the reader to \cite{fried} for detail of this method.
In this section $x:=(x_1,x_2, \cdots, x_d) \in  \mathbb{R}^d,\xi, \eta$ are defined the same way;  $t\in(0, T]$.  We use $\gamma$ to denote the H\"older exponents with respect to spatial variables. We can assume they are equal. For $\alpha\in(1, 2)$, we assume 
$$\gamma>2-\frac 2\alpha.$$
For $\alpha\in(\frac12,1)$, Chapter 4 gives the estimations the $Z$ and $Y$.
For $\alpha\in(1, 2)$, we need some lemmas before we can estimate $Z_1,  Z_2$ and $Y$.

From 
\cite{koch} we have
\begin{align*}
Z_j(t, x- \xi)=&Z_j^0(t, x- \xi, \xi)+V_{Z_j}(t, x; \xi), \qquad j=1, 2.\\
Y(t, x- \xi)=&Y_0(t, x- \xi, \xi)+V_Y(t, x; \xi).
\end{align*}
We refer the reader to \cite{koch} for the definitions of $Z_k^0(t, x- \xi, \xi), Y_0(t, x- \xi, \xi)$ and $V_Y(t, x; \xi).$
Here we list their estimations which we use to get the estimations of $Z_k$ and  $Y$ in section 3.
These estimations are given in section 2.2 of \cite{koch} or Lemma 15 in \cite{Pskhu09}.

\begin{lemma}\label{est.z_0}
$$|Z^0_1(t, x- \xi, \eta)|\leq Ct^{-\frac{\alpha d}{2}}  \mu_d(t^{-\frac{\alpha }{2}}|x-\xi|)p(t, x-\xi),$$
$$|Z^0_2(t, x- \xi, \eta)|\leq Ct^{-\frac{\alpha d}{2}+1}  \mu_d(t^{-\frac{\alpha }{2}}|x-\xi|)p(t, x-\xi),$$ where
\begin{equation}
\mu_d(z):= \left\{
  \begin{array}{ll}
   $$ 1$$, & \hbox{ $d=1$;} \\
       $$1+|\log z|$$,  & \hbox{ $d=2$;} \\
     $$z^{2-d} $$,  & \hbox{ $d\geq 3$.}
  \end{array}
\right. \label{1e.Z_0-bound}
\end{equation}
\end{lemma}

\begin{lemma}\label{1est.y_0}
$$|Y_0(t, x- \xi, \eta)|\leq Ct^{\alpha-\frac{\alpha d}{2}-1}  \mu_d(t^{-\frac{\alpha }{2}}|x-\xi|)p(t, x-\xi),$$ where
\begin{equation}
\mu_d(z):= \left\{
  \begin{array}{ll}
   $$ 1$$, & \hbox{ $d\leq3$;} \\
       $$1+|\log z|$$,  & \hbox{ $d=4$;} \\
     $$z^{4-d} $$,  & \hbox{ $d\geq 5$.}
  \end{array}
\right. \label{1e.Y_0-bound}
\end{equation}
\end{lemma}
\
The following estimations of $V_{Z_1},V_{Z_2}$ and $V_Y$ are from Theorem 1 of \cite{koch}, where $\nu_1\in (0,1)$, such that $\gamma>\nu_1>2- \frac{2}{\alpha}$ and $ \nu_0=\nu_1-2+\frac{2}{\alpha}$.

\begin{lemma}\label{1est.v_z_1}
\begin{equation}
|V_{Z_1}(t, x; \xi)|\leq 
\begin{cases} Ct^{(\gamma-1)\frac\alpha2}  p(t, x-\xi)\,, & \qquad \ d=1\,;\\
Ct^{\nu_0 \alpha-1}|x-\xi|^{-d+\gamma-\nu_1+2-\nu_0}  p(t, x-\xi)\,,
 &\qquad  \ d\geq 2  \\
\end{cases}
\end{equation} 
\end{lemma}

\begin{lemma}\label{1est.v_z_2}
\begin{equation}
|V_{Z_2}(t, x; \xi)|\leq 
\begin{cases} Ct^{(\gamma-1)\frac\alpha2+1}  p(t, x-\xi)\,, & \qquad \ d=1\,;\\
Ct^{\frac{\nu_0 \alpha}{2}+1-\alpha}|x-\xi|^{-d+\gamma-\nu_1+2-\nu_0}  p(t, x-\xi)\,,
 &\qquad  \ d\geq 2  \\
\end{cases}
\end{equation} 
\end{lemma}

\begin{lemma}\label{1est.v_y}
\begin{equation}
|V_Y(t, x; \xi)|\leq 
\begin{cases} Ct^{\alpha -1+(\gamma-1)\frac\alpha2}  p(t, x-\xi)\,, & \qquad \ d=1\,;\\
Ct^{\nu_0 \alpha-1}|x-\xi|^{-d+\gamma-\nu_1+2-\nu_0}  p(t, x-\xi)\,,
 &\qquad  \ d\geq 2  \\
\end{cases}
\end{equation} 
\end{lemma}

Based on the above three lemmas we have

\begin{lemma}\label{1est.y} 
Let $x \in  \mathbb{R}^d, t\in(0, T]$.   Then
\begin{equation}
|Y(t, x- \xi)|\leq \left\{
  \begin{array}{ll}
   $$ Ct^{-1+\frac{\alpha}{2}}   p(t, x-\xi)$$,  & \hbox{$d=1$;} \\
     $$Ct^{\alpha -\frac\alpha2\gamma+\nu_0\alpha-2} |x-\xi|^{-d+\gamma-2\nu_0+\frac{2}{\alpha}}  p(t, x-\xi)$$,  & \hbox{$d\geq 2$.}
  \end{array}
\right. \label{1e.Y-bound}
\end{equation}
\end{lemma}

\begin{proof} 

We "add" together the estimation of $Y_0$ in Lemma \ref{1est.y_0} and $V_y$ in Lemma \ref{1est.v_y} to get the estimation of $Y$.  We use the following inequality throughout the proof.
$$a, b,  \sigma>0, \quad \text{then} \quad \exists \sigma, C>0, \quad  s.t. \quad x^{a}e^{-\sigma x^b}<Ce^{-\sigma'x^b},  $$

First when $d=1$,

\begin{eqnarray*}
|Y(t, x-\xi)|
&\leq& |Y_0(t, x-\xi, \xi)|+|V_Y(t, x,\xi)|\\
&\leq& Ct^{\alpha -1+(\gamma-1)\frac\alpha2}  p(t, x-\xi)+Ct^{-1+\frac{\alpha}{2}}  p(t, x-\xi)\\
&\leq& Ct^{-1+\frac{\alpha}{2}}   p(t, x-\xi)\,.
\end{eqnarray*}

When $d\geq 5$,  
by the fact $$\nu_0=\nu_1-2+2/\alpha \qquad\text{and}\qquad \gamma>\nu_1>2-\frac2\alpha\,,$$ we have 
$$4-\gamma+2\nu_0-\frac2\alpha=-\gamma+2\nu_1+\frac{2}{\alpha}\geq0\,.$$
Therefore
\begin{eqnarray*} 
|Y_0(t, x-\xi, \xi)|
&\leq&Ct^{\alpha-\frac{\alpha d}{2}-1} \bigg|\frac{x-\xi}{t^\frac\alpha2}\bigg|^{4-d}  p(t, x-\xi)\\
&=& C t^{\alpha-\frac{\alpha d}{2}-1}\bigg|\frac{x-\xi}{t^\frac\alpha2}\bigg|^{-d+\gamma-2\nu_0+\frac{2}{\alpha} }\bigg|\frac{x-\xi}{t^\frac\alpha2}\bigg|^{4-\gamma+2\nu_0-\frac2\alpha}p(t, x-\xi)\\
&\leq& C t^{\alpha -\frac\alpha2\gamma+\nu_0\alpha-2}|x-\xi|^{-d+\gamma-2\nu_0+\frac{2}{\alpha} }p(t, x-\xi)\,.
\end{eqnarray*}

Furthermore because of the assumption $$ \gamma>2-\frac 2\alpha,$$ we have 
$$\alpha -\frac\alpha2\gamma+\nu_0\alpha-2<\nu_0\alpha-1.$$
Therefore
\begin{eqnarray*}
|Y(t, x-\xi)|
&\leq&|Y_0(t, x-\xi, \xi)|+|V_y(t, x,\xi)|\\
&\leq&  C t^{\alpha -\frac\alpha2\gamma+\nu_0\alpha-2}|x-\xi|^{-d+\gamma-2\nu_0+\frac{2}{\alpha} }p(t, x-\xi)\\
&+&Ct^{v_0\alpha-1} |x-\xi|^{-d+\gamma-2\nu_0+\frac{2}{\alpha}}  p(t, x-\xi)\\
&\leq& Ct^{\alpha -\frac\alpha2\gamma+\nu_0\alpha-2} |x-\xi|^{-d+\gamma-2\nu_0+\frac{2}{\alpha}} p(t, x-\xi) \,.
\end{eqnarray*}

When $d=2$ and $d=3$,  as in the previous cases, we first have 

\begin{eqnarray*} 
|Y_0(t, x-\xi, \xi)|
&\leq& Ct^{\alpha-\frac{\alpha n}{2}-1}  p(t, x-\xi)\\
&\leq& C t^{\alpha -\frac\alpha2\gamma+\nu_0\alpha-2}|x-\xi|^{-d+\gamma-2\nu_0+\frac{2}{\alpha}} p(t, x-\xi).
\end{eqnarray*}
Then as the last step in the case of  $d\geq 5$,  we have
\begin{eqnarray*}
|Y(t, x-\xi)|
&\leq& Ct^{\alpha -\frac\alpha2\gamma+\nu_0\alpha-2} |x-\xi|^{-d+\gamma-2\nu_0+\frac{2}{\alpha}} p(t, x-\xi) \,.
\end{eqnarray*}

When $d=4$,  let's first transform the estimation of $Y_0$ into the following form:
$$t^{\zeta_d} |x-\xi|^{\kappa_d}   p (t, x-\xi).$$
We have
\begin{eqnarray*} 
|Y_0(t, x-\xi, \xi)|
&\leq&Ct^{\alpha-\frac{\alpha d}{2}-1}  \left\{\bigg|\frac{x-\xi}{t^\frac\alpha2}\bigg|^{\theta}+\bigg|\frac{t^\frac\alpha2}{x-\xi}\bigg|^{\theta}\right\}  p(t, x-\xi)\\
&\leq&Ct^{\alpha-\frac{\alpha d}{2}-1}  \bigg|\frac{t^\frac\alpha2}{x-\xi}\bigg|^{\theta}\left\{\bigg|\frac{x-\xi}{t^\frac\alpha2}\bigg|^{2\theta}+1\right\}  p(t, x-\xi),\\
\end{eqnarray*}
for $\forall \, \theta>0$.

If $|\frac{x-\xi}{t^\frac\alpha2}|\leq1$, then
\begin{eqnarray*} 
\left\{\bigg|\frac{x-\xi}{t^\frac\alpha2}\bigg|^{2\theta}+1\right\}  p(t, x-\xi)
&\leq&2 p(t, x-\xi);
\end{eqnarray*}
if $|\frac{x-\xi}{t^\frac\alpha2}|>1$, then
\begin{eqnarray*} 
\left\{\bigg|\frac{x-\xi}{t^\frac\alpha2}\bigg|^{2\theta}+1\right\}  p(t, x-\xi)
&\leq&2\bigg|\frac{x-\xi}{t^\frac\alpha2}\bigg|^{2\theta}  p(t, x-\xi)\\
&\leq&Cp(t, x-\xi).
\end{eqnarray*}

Therefore if we choose $\theta>0$ such that  $$-\theta>-d+\gamma-2\nu_0+\frac{2}{\alpha},$$ we have
\begin{eqnarray*} 
|Y_0(t, x-\xi, \xi)|
&\leq&Ct^{\alpha-\frac{\alpha d}{2}-1}  \bigg|\frac{x-\xi}{t^\frac\alpha2}\bigg|^{-\theta}p(t, x-\xi)\\
&\leq& C t^{\alpha-\frac{\alpha d}{2}-1}\bigg|\frac{x-\xi}{t^\frac\alpha2}\bigg|^{-d+\gamma-2\nu_0+\frac{2}{\alpha} }p(t, x-\xi)\\
&\leq& C t^{\alpha -\frac\alpha2\gamma+\nu_0\alpha-2}|x-\xi|^{-d+\gamma-2\nu_0+\frac{2}{\alpha} }p(t, x-\xi)\,.
\end{eqnarray*}
As in previous two cases, we end up with
\begin{eqnarray*}
|Y(t, x- \xi)|
&\leq& Ct^{\alpha -\frac\alpha2\gamma+\nu_0\alpha-2} |x-\xi|^{-d+\gamma-2\nu_0+\frac{2}{\alpha}} p(t, x-\xi) \,.
\end{eqnarray*}

\end{proof}
Let's denote the the estimation function of $Y$ by $t^{\zeta_d} |x-\xi|^{\kappa_d}   p (t, x-\xi).$ For the estimation of integral  \eqref{basin} involving $Y$ and fractional kernel   it more convenient  to represent the estimation of $Y$ as the 
the product of one dimensional functions. To this purpose, 
as in the case of $0<\alpha<1$, the estimation of $Y$ is represented as the product of one dimensional functions, which is shown in the following lemma.
\begin{lemma}
Let $x_i, \xi_i \in \RR, t\in(0, T]$
\begin{equation}\label{e.Y-product-b}
|Y(t, x- \xi)|\leq C\prod_{i=1}^{d} t^{\zeta_d/d} |x_i-\xi_i|^{\kappa_d/d}   p (t, x_i-\xi_i)\,,
\end{equation}
where  $\zeta_{d}$ and $\kappa_{d}$ are the powers of $t$ and $x-\xi$ in the estimation of $Y$, i.e.,
\begin{equation}\label{zeta2}
 \zeta_{d} =
\begin{cases} -1+\frac{\alpha}{2}  ,  &  d=1 ;  \\
         \alpha -\frac\alpha2\gamma+\nu_0\alpha-2  ,  &  d\geq 2\,. 
  \end{cases}
  \end{equation}
  and
\begin{equation}
 \kappa_{d} =
\begin{cases}   0 ,  &  d=1  ;  \\
     -d+\gamma-2\nu_0+\frac{2}{\alpha}  ,  &    d\geq 2\,. 
  \end{cases}\label{kappa2}
\end{equation}
\end{lemma}
\begin{lemma}\label{est.z_k} 
\[
\sup_{t, x}\left|\int_{\mathbb{R}^d}Z_{k+1}(t, x-\xi)u_{k}(t, \xi)d\xi\right| \leq C\, \qquad k=0, 1.
\]
\end{lemma}
\begin{proof}
First recall that $u_{k}(x)$ are bounded. Thanks to the following fact from \cite{Koc}
$$\int_{\mathbb{R}^d}Z_1^0(t, x, \xi)d\xi =1 \quad \text{and}\quad \int_{\mathbb{R}^d}Z_2^0(t, x, \xi)d\xi =t ,$$ we only need to show
$$\sup_{t, x}\int_{\mathbb{R}^d}V_{Z_{j}}(t, x, \xi)\ud\xi \leq C,$$ since $u_k$ are bounded.

Let's consider the case $d\geq3$ and $d=2$ for $V_{z_1}$ as examples. When $d\geq3$, 
by the estimation of $V_{Z_1}$ in Lemma \ref{est.z_0},  we have
\begin{eqnarray*}
\int_{\mathbb{R}^d}|V_{Z_1}(t, x, \xi) |d\xi
&\leq&  \int_{\mathbb{R}^d}   Ct^{-\frac{\alpha d}{2}}  \mu_d(t^{-\frac{\alpha }{2}}|x-\xi|)p(t, x-\xi) dy\\
&\leq&  \int_{\mathbb{R}^d}   Ct^{-\frac{\alpha d}{2}+d}  \mu_d(z)p(t, x-\xi) dz\\
&\leq& Ct^{-\frac{\alpha d}{2}+d}\\
&\leq& C,
\end{eqnarray*}
due to the fact $ t\in(0, T].$
\

For the case $d=2, Z_1$, notice that $$\forall \theta>0, \exists C>0 \quad s.t. \quad (\log|z|+1)<c|z|^{\theta},  $$ as shown in the case of d=4 in the proof of \ref{1est.y}. Then the above argument ends proof.
The proof for the rest of the cases is almost the same, so we omit it.

\end{proof}

\section{Miscellaneous estimations}
\
For fractional kernel, we need the following estimation, which is immediate from Corollary 15 of \cite{HuHu15}.
\begin{lemma}\label{basic.ineq1}
Let $0<r, s\le T$ and 
\begin{equation}\label{condtion.h2}
 2H_i+\frac{2\kappa_d}{d}>0.
 \end{equation}
 Then for any
$\rho_1, \tau_2 \in \mathbb{R},  \rho_1\neq\tau_2$,   we have
\begin{eqnarray*}
\int_{\mathbb{R}^2}|\rho_1-\tau_1|^{2H_i-2}|\rho_2-\rho_1|^{\frac{\kappa_d}{d}}|\tau_2-\tau_1|^{\frac{\kappa_d}{d}}p(s, \rho_2-\rho_1)p(r, \tau_2-\tau_1)\ud\rho_1 \ud\tau_1\leq C (s\: r)^{\theta_i},
\end{eqnarray*} where
$$\theta_i=
\begin{cases}
C (s\: r)^{\frac{H_i d+\kappa_d}{2d}\alpha}\,,& \qquad 2H_i-2+\kappa_d/d\not =-1;\\
C (s\: r)^{\frac{d\epsilon+\kappa_d+d }{4d}\alpha}\,,&  \qquad2H_i-2+\kappa_d/d= -1\,.
\end{cases} \nonumber\\$$

\end{lemma}
\begin{proof}
In Corollary 15 of \cite{HuHu15}, let $\theta_1=2H_i-2, \theta_2=\kappa_d/d$. Then notice that for $0<r\le T$
$$\forall \epsilon<0, \exists\: C>0, \quad s.t. \quad \log r<Cr^\epsilon. $$
\end{proof}

The next lemma can be proved as in Lemma 11 of \cite{HuHu15}.
\begin{lemma}\label{intes}
Let $-1<\beta\leq0,  x\in \mathbb{\RR}^d$.    Then, there is a constant $C$, dependent on
$\sigma$,  $\alpha$ and $\beta$,  but independent of $\xi$ and $s$ such that
\[
\int_{\mathbb{R}^d}|x|^{\beta}
 p(s, x-\xi)d x\leq C s^{\frac{\alpha\beta}{2}+\frac{\alpha}{2}d}\,.
\]
\end{lemma}
For Bessel kernel, we need the following lemma.
\begin{lemma}\label{basic.ineq2}
Assume $0<s, r\leq T$ and $y_1, y_2,  z_1, z_2\in\RR^d$, we have that
\[
\int_{\RR^{2d}} \left|Y(r, y_1-y_2) Y(s, z_1-z_2)\right|\: \int_0^\infty \omega^{-\frac{\kappa}{2}-1}e^{-\omega} e^{\frac{-|y_1-z_1|^2}{4\omega}}d\omega \ud y_1 \ud z_1
\leq C \cdot(r\: s)^\ell,
\] where 
\[
\ell:=\zeta_d-\frac{\alpha}{4}\kappa+\frac{\alpha}{2}\kappa_d+\frac{\alpha}{2}d
\]
\end{lemma}
\begin{proof}
Recall that the estimation of $Y(t, x)$ in Lemma 9 of  \cite{HuHu15} and \eqref{1e.Y-bound} has the following form:
$$Cs^{\zeta_d}|x|^{\kappa_d}p(t, x).$$
By substituting $Y$, we have
\begin{align*}
\int_{\RR^{2d}} \left|Y(r, y_1-y_2)Y(s, z_1-z_2)\right|\: \int_0^\infty \omega^{-\frac{\kappa}{2}-1}e^{-\omega} e^{\frac{-|y_1-z_1|^2}{4\omega}}d\omega \ud y_1 \ud z_1
\end{align*}
$$\leq C\int_{\RR^{d}}s^{\zeta_d}|z_2-z_1^{\kappa_d}|p(s, z_2-z_1)r^{\zeta_d}\int_0^\infty I\:\cdot \omega^{-\frac{\kappa}{2}-1}e^{-\omega} \ud\omega \ud z_1,$$

where
$$I:=\int_{\RR_d}|y_2-y_1|^{\kappa_d} \exp \left\{-\sigma \left|\frac {y_2-y_1}{r^{\frac\alpha2}}\right|^{\frac{2}{2-\alpha}}\right\}\exp \left\{-\frac {|y_1-z_1|^2}{4\omega}\right\}\ud y_1.$$

For $I$, we have two estimations:
\begin{align*}
I&\leq \int_{\RR_d}|y_2-y_1|^{\kappa_d} \exp \left\{-\sigma \left|\frac {y_2-y_1}{r^{\frac\alpha2}}\right|^{\frac{2}{2-\alpha}}\right\}\ud y_1\\
&\leq C r^{\frac{\alpha}{2}\kappa_d+\frac{\alpha}{2}d},
\end{align*}
and
\begin{align*}
I&\leq \int_{\RR_d}|y_2-y_1|^{\kappa_d} \exp \left\{-\frac {|y_1-z_1|^2}{4\omega}\right\}\ud y_1\\
&\leq C \omega^{\frac{\kappa_d}{2}+\frac{d}{2}},
\end{align*}
by  Lemma \ref{intes}.

With the estimations of $I$, we have 
\begin{align*}
\int_0^{\infty} I\:\cdot \omega^{-\frac{\kappa}{2}-1}e^{-\omega}\ud\omega&=\int_0^{r^\alpha} I\:\cdot \omega^{-\frac{\kappa}{2}-1}e^{-\omega} \ud\omega+\int_{r^\alpha}^\infty I\:\cdot \omega^{-\frac{\kappa}{2}-1}e^{-\omega}\ud\omega\\
&\leq r^{\frac{\alpha}{2}\kappa_d+\frac{\alpha}{2}d-\frac{\alpha}{2}\kappa}+ \int^\infty_{r^\alpha}\omega^{\frac{\kappa_d}{2}+\frac{d}{2}}\: \omega^{-\frac{\kappa}{2}-1}e^{-\omega}\ud\omega.
\end{align*}

For $\int^\infty_{r^\alpha}\omega^{\frac{\kappa_d}{2}+\frac{d}{2}}\: \omega^{-\frac{\kappa}{2}-1}e^{-\omega}\ud\omega$, 

\

if $\frac{\kappa_d}{2}+\frac{d}{2}-\frac{\kappa}{2}<0$
\begin{align*}
\int^\infty_{r^\alpha}  \omega^{\frac{\kappa}{2}+\frac{d}{2}}\: \omega^{-\frac{\kappa_d}{2}-1}e^{-\omega}\ud\omega&\leq \int^\infty_{r^\alpha}  \omega^{\frac{\kappa_d}{2}+\frac{d}{2}-\frac{\kappa}{2}-1}\ud\omega\\
&=C r^{\alpha(\frac{\kappa_d}{2}+\frac{d}{2}-\frac{\kappa}{2})};
\end{align*}

if $\frac{\kappa_d}{2}+\frac{d}{2}-\frac{\kappa}{2}\geq0$
\begin{align*}
\int^\infty_{r^\alpha}  \omega^{\frac{\kappa}{2}+\frac{d}{2}}\: \omega^{-\frac{\kappa_d}{2}-1}e^{-\omega}\ud\omega&= \int^\infty_{r^\alpha}  \omega^{\frac{\kappa}{2}+\frac{d}{2}-\frac{\kappa_d}{2}-1}e^{-\omega}\ud\omega\\
&= C\\
&\leq C r^{\alpha(\frac{\kappa_d}{2}+\frac{d}{2}-\frac{\kappa}{2})}.
\end{align*}

Therefore we end up with 
$$\int_0^\infty I\:\cdot \omega^{-\frac{\kappa}{2}-1}e^{-\omega} \ud\omega\leq C r^{\alpha(\frac{\kappa_d}{2}+\frac{d}{2}-\frac{\kappa}{2})}.$$

The estimation of integration with respect to $z_1$ is straightforward thank to fact that C is independent of $z_1$.

We have
$$\int_{\RR^{d}}s^{\zeta_d}|z_2-z_1|p(s, z_2-z_1)r^{\zeta_d}\int_0^\infty I\:\cdot \omega^{-\frac{\kappa}{2}-1}e^{-\omega} \ud\omega \ud z_1 $$
$$\leq Cr^{\alpha(\frac{\kappa_d}{2}+\frac{d}{2}-\frac{\kappa}{2})}\cdot r^{\zeta_d}\cdot s^{\alpha(\frac{\kappa_d}{2}+\frac{d}{2}-\frac{\kappa}{2})} s^{\zeta_d},$$
by  Lemma \ref{intes}.
\

By symmetry,  we have
$$\int_{\RR^{2d}} \left|Y(r, y_1-y_2) Y(s, z_1-z_2)\right|\: \int_0^\infty \omega^{-\frac{\kappa}{2}-1}e^{-\omega} e^{\frac{-|y_1-z_1|^2}{4\omega}}d\omega \ud y_1 \ud z_1 $$
$$\leq C s^{\alpha(\frac{\kappa_d}{2}+\frac{d}{2}-\frac{\kappa}{2})}\cdot s^{\zeta_d}\cdot r^{\alpha(\frac{\kappa_d}{2}+\frac{d}{2}-\frac{\kappa}{2})} r^{\zeta_d}.$$

Combining the two estimations we get the estimation in the lemma.
\end{proof}

The following lemma is Theorem 3.5 from \cite{BC3}.
\begin{lemma}
\label{stint}
Let $T_n(t)=\{s=(s_1, \ldots,s_n):\: 0<s_1<s_2 < \ldots <s_n<t\}$. Then 
\[
 \int_{T_n(t)} [(t-s_n)(s_n-s_{n-1}) \ldots
(s_2-s_1)]^{h} \ud  s =\frac{\Gamma(1+h)^{n}}{\Gamma(n(1+h)+1)}
t^{n(1+h)},
\]
if and only if $1+h>0$.
\end{lemma}
\section{Existence and uniqueness of the solution}
\begin{theorem}\label{main.theorem}
Assume the following conditions:
\begin{enumerate}[(1)]
\item $\lambda(t)$ is a nonnegative definite locally integrable function\,;
\item $\alpha \in (1/2,1)\cup(1, 2)$.
\end{enumerate}
Then relation \eref{eq: L2 chaos} holds for each $(t,x)$, if  any of the following is true.
Consequently, equation \eref{E:SPDE} admits a unique mild solution in the sense of Definition \ref{D:Sol}.
\begin{enumerate}[(i)]

\item $\Lambda(x)$ is fractional kernel with condition:
\begin{equation*}
H_i>\begin{cases}
\frac12,  &\qquad \hbox{}\  \ d=1, 2, 3, 4\\
1-\frac{2}{d}-\frac{\gamma}{2d},&\qquad \hbox{}\ \ d\ge 5, \alpha\in(0, 1) \\
1-\frac{2}{d},&\qquad \hbox{}\ \ d\ge 5, \alpha\in(1, 2) \\

\end{cases}  
\label{e.cond.main-hi}
 \end{equation*}
 and
 \begin{equation*}
 \sum_{i=1}^{d}H_i>d-2+\frac1\alpha \,. 
\label{e.cond.main}
 \end{equation*}
\item $\Lambda(x)$ is the Reisz or Bessel  kernel  with condition:
\[
\kappa< 4-2/\alpha;
\]
\end{enumerate}
\end{theorem}
\begin{proof}

Fix $t>0$ and $x\in \mathbb{R}^d$.

Let $$ (s, y, t, x):=(s_1, y_1,\cdots, s_n, y_n, t, x);$$
\[
g_n(s, y, t, x):=\frac{1}{n!}Y(t-s_{\sigma(n)}, x-y_{\sigma(n)} )\cdots Y(s_{\sigma(2)}-s_{\sigma(1)}, y_{\sigma(2)}-y_{\sigma(1)}) \,;
\]
\begin{align*}
 f_n(s, y, t, x):=g_n(s, y, t, x)J_0(s_{\sigma(1)}, x_{\sigma(1)}),
\end{align*}
where $\sigma$ denotes a permutation of  $ \{1,2,\cdots, n\}$ such that
$0<s_{\sigma(1)}<\cdots<s_{\sigma(n)}<t$.

By iteration of $u(t, x)$, we have 
\begin{multline}\label{E:fnNorm}
n! \| f_n(\cdot,\cdot,t,x)\|^2_{\mathcal{H}^{\otimes n}}\\
= n!\int_{[0,t]^{2n}} \ud s\ud r\int_{\mathbb{R}^{2nd}}\ud y\ud z\: f_n(s, y, t, x)f_n(r, z, t, x)\prod_{i=1}^n\Lambda(y_i-z_i)\prod_{i=1}^n\lambda(s_i-r_i),
\end{multline}
where $\ud y:=\ud y_1 \cdots \ud y_n$, the differentials $\ud z$, $\ud s$ and $\ud r$ are defined similarly.
Set $\mu(\ud\xi) : = \prod_{i=1}^{n} \mu(\ud\xi_i)$.

Recall that $J_0$ is bounded, so we have

\begin{multline*}
n! \| f_n(\cdot,\cdot,t,x)\|^2_{\mathcal{H}^{\otimes n}}\\
\leq C\frac{1}{n!}\int_{[0,t]^{2n}} \ud s\ud r\int_{\mathbb{R}^{2nd}}\ud y\ud z\: g_n(s, y, t, x)g_n(r, z, t, x)\prod_{i=1}^n\Lambda(y_i-z_i)\prod_{i=1}^n\lambda(s_i-r_i).
\end{multline*}
Furthermore by Cauchy-Schwarz inequality,

\begin{multline*}
\int_{\mathbb{R}^{2nd}}\ud y\ud z\: g_n(s, y, t, x)g_n(r, z, t, x)\prod_{i=1}^n\Lambda(y_i-z_i)\\
\leq \left\{\int_{\mathbb{R}^{2nd}}\ud y\ud z\: g_n(s, y, t, x)g_n(s, z, t, x)\prod_{i=1}^n\Lambda(y_i-z_i)\right\}^{1/2}\\
\cdot\left\{\int_{\mathbb{R}^{2nd}}\ud y\ud z\: g_n(r, y, t, x)g_n(r, z, t, x)\prod_{i=1}^n\Lambda(y_i-z_i)\right\}^{1/2}
\end{multline*}

(i)
Let $\Lambda(\cdot)=\varphi_{H}(\cdot)$ and use the estimation of $Y$ in Lemma \ref{1est.y}, we have
\begin{multline}\label{basin}
\int_{\mathbb{R}^{2nd}}\ud y\ud z\: g_n(s, y, t, x)g_n(s, z, t, x)\prod_{i=1}^n\Lambda(y_i-z_i)\\
\leq C \prod_{i=1}^d\int_{\mathbb{R}^{2n}}\prod_{k=1}^n\varphi_{H_i}(y_{ik}-z_{ik}) \Theta_{n}(t, y_{ik}, s)\Theta_{n}(t, z_{ik}, s)dy_idz_i
\end{multline}
where
$$
\Theta_{ n}(t, y_{ik}, s):= |s_{\sigma(k+1)}-s_{\sigma(k)}|^{\frac{\zeta_{d}}{d}}|y_{i\sigma(k+1)}-y_{i\sigma(k)}|^{\frac{\kappa_{d}}{d}}  p(s_{\sigma(k+1)}-s_{\sigma(k)}, y_{i\sigma(k+1)}-y_{i\sigma(k)});
$$

$$y_i=(y_{i1}, y_{i2}, \cdots, y_{ik}, \cdots,y_{in}), \qquad z_i=(z_{i1}, z_{i2}, \cdots, z_{ik}, \cdots,z_{in});$$

$$dy_i:=\prod_{k=1}^n dy_{ik}\qquad dz_i:=\prod_{k=1}^n dz_{ik}; $$

 and
 $$y_{\sigma(k+1)}=z_{\sigma(k+1)}:=x_{i}\,; \qquad s_{\sigma(n+1)}=r_{\sigma(n+1)}:=t.$$ 
 
Let's first consider the case $2H_i-2+\kappa_d/d\ne -1$. Applying Lemma \ref{basic.ineq1} to 
\begin{equation}\label{theta.int}
\int_{\mathbb{R}^{2n}}\prod_{k=1}^n\varphi_{H_i}(y_{ik}-z_{ik}) \Theta_{n}(t, y_{ik}, s)\Theta_{n}(t, z_{ik}, s)dy_idz_i
\end{equation}
for $dy_{i\sigma(1)}dz_{i\sigma(1)}$, we have

\begin{multline*}
\int_{\mathbb{R}^{2n}}\prod_{k=1}^n\varphi_{H_i}(y_{ik}-z_{ik}) \Theta_{n}(t, y_{ik}, s)\Theta_{n}(t, z_{ik}, s)dy_idz_i\\
\leq C(s_{i\sigma(2)}-s_{i\sigma(1)})^{2\ell_i}\int_{\mathbb{R}^{2n}}\prod_{k=2}^n\varphi_{H_i}(y_{ik}-z_{ik}) \Theta_{n}(t, y_{ik}, s)\Theta_{n}(t, z_{ik}, s)dy_idz_i
\end{multline*}
where  $$ \ell_i=\frac{\zeta_{d}}{d}+\theta_i \,.$$

Applying Lemma \ref{basic.ineq1} to \eqref{theta.int} for $dy_{i\sigma(k)}dz_{i\sigma(k)}, k=2,\cdots, n$,
we have 
$$ \prod_{i=1}^d\int_{\mathbb{R}^{2n}}\prod_{k=1}^n\varphi_{H_i}(y_{ik}-z_{ik}) \Theta_{n}(t, y_{ik}, s)\Theta_{n}(t, z_{ik}, s)dy_idz_i \leq \prod_{k=1}^nC^n(s_{\sigma(k+1)}-s_{\sigma(k)})^{2\ell}$$
where
\begin{equation}\label{de.l}
\ell =\sum_{i=1}^d \ell_i= \zeta_{d} + \frac{|H|\alpha}{2}     +\frac{\kappa_d\alpha }{2 }   \quad {\rm with}\quad
|H|=\sum_{i=1}^d H_i\,.
\end{equation}

Due to the same argument, we have
$$ \prod_{i=1}^d\int_{\mathbb{R}^{2n}}\prod_{k=1}^n\varphi_{H_i}(y_{ik}-z_{ik}) \Theta_{n}(t, y_{ik}, r)\Theta_{n}(t, z_{ik}, r)dy_idz_i \leq \prod_{k=1}^nC^n(r_{\rho(k+1)}-r_{\rho(k)})^{2\ell}$$
Therefore
\[
\int_{\mathbb{R}^{2nd}}\ud y\ud z\: g_n(s, y, t, x)g_n(r, z, t, x)\prod_{i=1}^n\Lambda(y_i-z_i) \leq C^n \; (\phi(s)\phi(r))^\ell,
\]
where 
\[
 \phi(s) :=\prod_{i=1}^{n}(s_{\sigma(i+1)}- s_{\sigma(i)}), \qquad \phi(r): = \prod_{i=1}^{n} (r_{\rho(i+1)}- r_{\rho(i)}),
\] 
with
\[ 
0<s_{\sigma(1)}<s_{\sigma(2)}<
\ldots < s_{\sigma(n)} \quad \text{and} \quad 0<r_{\rho(1)}<r_{\rho(2)}< \ldots <
r_{\rho(n)} .
\]
Hence,
\begin{align*}
n! \| f_n(\cdot,\cdot,t,x)\|^2_{\mathcal{H}^{\otimes n}} 
&\leq  \frac{C^n}{n!}\int_{[0,t]^{2n}} \prod_{i=1}^n \lambda(s_i-r_i) (\phi(s)\phi(r))^\ell    \ud s \ud r\\
&\leq  \frac{C^n}{n!} \frac{1}{2} \int_{[0,t]^{2n}} \prod_{i=1}^n \lambda(s_i-r_i) \left(\phi(s) ^{2\ell} + \phi(r)^{2\ell} \right)  \ud s \ud r\\
&=  \frac{C^n}{n!} \int_{[0,t]^{2n}} \prod_{i=1}^n \lambda(s_i-r_i) \phi(s) ^{2\ell} \ud s \ud r\\
&\leq \frac{C^nC_t^n}{n!} \int_{[0,t]^n} \phi(s)^{2\ell} \ud s\\
&=  C^nC_t^n  \int_{T_n(t)} \phi(s)^{2\ell} \ud s\\
&=  \frac{C^n C_t^n \Gamma(2\ell +1)^n t^{(2\ell+1)n}}{\Gamma((2\ell+1)n+1)}\,,
\end{align*}
where $C_t= 2\int_0^t \lambda(r)dr$. The last step is by Lemma \ref{stint}.

Therefore, 
\[
n! \| f_n(\cdot,\cdot,t,x)\|^2_{\mathcal{H}^{\otimes n}}\leq  \frac {C^nC_t^n}{\Gamma((2\ell+1)n+1)}\:,
\]
and 
$\sum_{n\ge 0} n! \| f_n(\cdot,\cdot,t,x)\|^2_{\mathcal{H}^{\otimes n}}$ converges if $\ell>-1/2$.

\

Next we need to show $$\ell>-1/2  \iff  |H|>d-2+\frac{1}{\alpha}\,.$$

Firstly by definition of $\ell$, \eqref{de.l}
$$\ell>-1/2\iff |H|>-\frac1\alpha-\kappa_d-\frac{2}{\alpha}\zeta_d.$$

Then using the definition of $\zeta_d$ and $\kappa_d$ in (4.2), (4.3) of \cite{HuHu15} for $1/2<\alpha<1$ and \eqref{zeta2}, \eqref{kappa2} for $1<\alpha<2$, we have: 

when $1/2<\alpha<1$,
\begin{equation*}
\frac1\alpha-\kappa_d-\frac{2}{\alpha}\zeta_d =
\begin{cases} -1+\frac1\alpha  ,  &  d=1 ;  \\
      \frac1\alpha   ,  &  d=2 ;  \\
       \frac1\alpha+2 ,  & d=4;  \\
     \frac1\alpha-2+d  ,  &  d=3\ {\rm or}\ d\geq 5\,;
  \end{cases}
  \end{equation*}

when $1<\alpha<2$,
\begin{equation*}
\frac1\alpha-\kappa_d-\frac{2}{\alpha}\zeta_d =
\begin{cases} -1+\frac1\alpha  ,  &  d=1 ;  \\
       d-2+\frac{1}{\alpha} ,  &  d\geq 2\,;
  \end{cases}
  \end{equation*}

For case $2H_i-2+\kappa_d/d= -1$, 
 applying Lemma \ref{basic.ineq1} to \eqref{theta.int},
we have 
$$ \prod_{i=1}^d\int_{\mathbb{R}^{2n}}\prod_{k=1}^n\varphi_{H_i}(y_{ik}-z_{ik}) \Theta_{n}(t, y_{ik}, s)\Theta_{n}(t, z_{ik}, s)dy_idz_i \leq \prod_{k=1}^nC^n(s_{\sigma(k+1)}-s_{\sigma(k)})^{2\ell'},$$
where
\begin{equation*}\label{de.l2}
\ell' = \zeta_{d} + \frac{d\epsilon+\kappa_d+d }{4}\alpha  \quad {\rm with}\quad
|H|=\sum_{i=1}^d H_i\,.
\end{equation*}

Using the relation $2H_i-2+\kappa_d/d= -1$, we have
$$\ell'= \ell + \frac{d\alpha}{4}\varepsilon.$$
Since
$$|H|>d-2+\frac{1}{\alpha}  \implies \ell>-1/2\,,$$
we can choose $\varepsilon$ big enough such such
$$|H|>d-2+\frac{1}{\alpha} \implies \ell'>-1/2 \,.$$
\

Lastly, when $\alpha \in(1/2, 1)$, for $d\leq 4, H_i>1/2$ implies condition \eqref{condtion.h2}; for $d>4$, condition \eqref{condtion.h2} is implied by
 $$H_i>1-\frac{2}{d}-\frac{\gamma}{2d} $$ with $\gamma_0$ sufficiently small;
 when $\alpha \in(1, 2)$ for $d=1, H_i>1/2$ implies \eqref{condtion.h2}; for $d\geq 2$, \eqref{condtion.h2} is implied by
 $$H_i>1-\frac{2}{d} $$ with $\nu_0$ sufficiently small.
This completes the proof of Theorem \ref{main.theorem} for case of $\Lambda(\cdot)=\varphi_{H}(\cdot)$

\

(ii) Let $x=(x_1, x_2, \cdots, x_d)\in\RR^d$. For Reisz kernel, notice that 
$$|x|^{-\kappa}\leq C\prod_{i=1}^d|x_i|^{\frac{\kappa}{d}},$$ so this case is reduced to case (i) with $H_i=(-\frac{\kappa}{d}+2)\frac{1}{2}, \: i=1, 2, \cdots, d$.
\

Correspondingly
$$|H|>d-2+\frac{1}{\alpha} $$ is 
$$\kappa< 4-2/\alpha,$$ which also guarantees condition \eqref{condtion.h2} .

\

For Bessel kernel,
applying Lemma \ref{basic.ineq2} for $\ud y_{\sigma(i)}\ud z_{\sigma(i)}$ in the order of  $i=1, 2, \cdots, n$ to  
$$\int_{\mathbb{R}^{2nd}}\ud y\ud z\: g_n(s, y, t, x)g_n(s, z, t, x)\prod_{i=1}^n\Lambda(y_i-z_i)$$

yields 

$$\int_{\mathbb{R}^{2nd}}\ud y\ud z\: g_n(s, y, t, x)g_n(s, z, t, x)\prod_{i=1}^n\Lambda(y_i-z_i)\leq \prod_{k=1}^nC^n(s_{\sigma(k+1)}-s_{\sigma(k)})^{2\ell},$$
where
\[
\ell:=\zeta_d-\frac{\alpha}{4}\kappa+\frac{\alpha}{2}\kappa_d+\frac{\alpha}{2}d
\]

As in case (i), $\sum_{n\ge 0} n! \| f_n(\cdot,\cdot,t,x)\|^2_{\mathcal{H}^{\otimes n}}$ converges if $\ell>-1/2$. Then using the definition of  $\zeta_d$ and $\kappa_d$ in (4.2), (4.3) of \cite{HuHu15} for $1/2<\alpha<1$ and \eqref{zeta2}, \eqref{kappa2} for $1<\alpha<2$, we have
$$\ell>-1/2\iff\kappa< 4-2/\alpha.$$
This finishes the the proof of the theorem.
\end{proof}



\bibliographystyle{plain}



\end{document}